\documentclass[11pt]{amsart}

\newtheorem{theorem}{Theorem}
\newtheorem{lemma}{Lemma}

\newtheorem{corollary}{Corollary}

\theoremstyle{definition}

\theoremstyle{remark}
\newtheorem{remark}{Remark}

\numberwithin{equation}{section}

\begin{document}

\title[Order continuous extensions of positive compact operators]
{Order continuous extensions of  positive compact operators on Banach lattices}

\author[J.X. Chen]
{Jin Xi Chen }

\address{College of Mathematics and Information Science, Shaanxi Normal
University, Xi'an 710062, P.R. China}
\curraddr{Department of Mathematics, Southwest Jiaotong
University, Chengdu 610031, P.R. China}
 \email{jinxichen@home.swjtu.edu.cn}

\author[Z.L. Chen]
{Zi Li Chen}
\address{Department of Mathematics, Southwest Jiaotong
University, Chengdu 610031, P.R. China}
\email{zlchen@home.swjtu.edu.cn}
\thanks{The first two authors were supported in part by the Fundamental Research Funds for the Central Universities (SWJTU09ZT36)}
\author[G.X. Ji]
{Guo Xing Ji}
\address{College of Mathematics and Information Science, Shaanxi Normal
University, Xi'an 710062, P.R. China}
\email{gxji@snnu.edu.cn}
\thanks{The third author was supported in part by the NNSF of China (No.10971123)
and the Specialized Research Fund for the Doctoral Program of Higher Education (No.  20090202110001)}


\begin{abstract}
Let $E$ and $F$ be Banach lattices. Let $G$ be a vector sublattice of $E$ and $T: G\rightarrow F$ be an order continuous positive compact (resp. weakly compact) operators. We show that if $G$ is an ideal or an order dense sublattice of $E$, then $T$ has a norm preserving compact (resp. weakly compact) positive extension to $E$ which is likewise  order continuous on $E$. In particular, we prove that every  compact positive orthomorphism on an order dense sublattice of $E$ extends uniquely to a  compact positive orthomorphism on $E$.

\end{abstract}

\maketitle
\baselineskip 5mm

\section{Introduction}
In the operator version of Hahn-Banach-Kantorovich theorem, the range space
is assumed to be Dedekind complete, i.e., order complete (cf. \cite[Theorem 2.1]{2}). In the past two decades this assumption has been considerably weakened in a few papers, among which are those of Abramovich and Wickstead \cite{AW}, N. D\u{a}net \cite{D}, N. D\u{a}net and R. M. D\u{a}net \cite{DD}, R. M. D\u{a}net and Wong \cite{DW}, and Zhang \cite{Z}. Specially, we would like to mention the paper \cite{Z}. Let $E$ and $F$ be arbitrary Banach lattices and $G$ be an ideal or a majorizing sublattice of $E$. Zhang proved that every positive compact operator $T: G\rightarrow F$ has a positive compact extension to the whole  space $E$.
 \par Note that order continuous operators are a very important class of order bounded linear operators on Riesz spaces (i.e., vector lattices).
Recall that a positive operator $T: E\rightarrow F$ between two Riesz spaces is \textit{order continuous} if and only if $x_{\alpha}\downarrow 0$ in $E$ implies $Tx_{\alpha}\downarrow 0$ in $F$ (and also if and only if $0\leq x_{\alpha}\uparrow x$ in $E$ implies $Tx_{\alpha}\uparrow Tx$ in $F$). A well known result concerning the order continuous extension is that every positive order continuous operator from an order dense majorizing Riesz subspace of a Riesz space $E$ into a Dedekind complete Riesz space has a unique order continuous positive extension to all of $E$ (cf. \cite[Theorem 4.12]{2}). Other related extension results of this aspect were obtained with the range space Dedekind complete. See, e.g., \cite{A, HHS}, and \cite{SZ} specially for order continuous functionals. In \cite{SZ} it was shown that every norm bounded  order continuous positive linear functional defined on an ideal $I$ of a Banach lattice $E$ has a norm preserving order continuous positive extension to all of $E$. Clearly, every norm bounded linear functional is a compact operator. It should be noted that a compact (resp. weakly compact) operator between Banach lattices need not be an order continuous operator. Indeed, from Theorem 39.6 of \cite{Za} it is easy to see that whenever a  Banach lattice $E$ does not have order continuous norm, then there exists a positive (norm bounded) linear functional on $E$, which is not order continuous. On the other hand, an order continuous operator is not necessarily weakly compact. For instance, the identity operator $I:E\rightarrow E$ on a non-reflexive Banach lattice $E$ is a positive order continuous operator, which is not weakly compact.
\par The main purpose of this paper is to investigate order continuity of positive compact extension on Banach lattices. We prove that every order continuous positive compact (resp. weakly compact) operator from an ideal or an order dense sublattice $G$ of a Banach lattice $E$ into a Banach lattice $F$ has a norm preserving compact (resp. weakly compact) positive extension to $E$ which is likewise order continuous on $E$. In particular, we  prove that every  compact positive orthomorphism on an order dense sublattice $G$ of $E$ extends uniquely to a  compact positive orthomorphism on $E$. In our proofs we do not have to  assume that the range space is Dedekind complete.
\par Our notions are standard. For the theory of Banach lattices and positive operators, we refer the reader to the monographs \cite{2, M, HH, Za}.

\section{Order continuous extensions of positive compact operators}
Let $E$ be an Archimedean Riesz space and $G$ be a Riesz subspace (i.e., sublattice) of $E$. Recall that $G$ is said to be order dense in $E$ whenever for each $0<x\in E$, there exists some $y\in G$ with $0<y\leq x.$ It is useful to know that $G$ is order dense in $E$ if and only if $\{y\in G:0\leq y\leq x\}\uparrow x$ holds in $E$ for all $x\in E^{+}$ ( see \cite[Theorem 3.1]{2}). We say that $G$ is a majorizing subspace whenever for each $x\in E$ there exists some $y\in G$ with $x\leq y$. It should be noted that an order dense vector sublattice of $E$ need not be  a majorizing sublattice or an ideal of $E$. For instance, $c_{0}\oplus c$ is order dense in $\ell_{\infty}\oplus\ell_{\infty}$, but it is not a majorizing sublattice or an ideal of $\ell_{\infty}\oplus\ell_{\infty}$.
\par To prove our results we need the following useful lemma. See  \cite[Theorem 3.2]{2} or \cite[Theorem 23.2]{Za} for details.

\begin{lemma}\label{Lemma1}
Let $G$ be either an ideal or an order dense Riesz subspace of a Riesz space $E$,and let $D$ be an upwards directed subset of $G$. Let $x\in G$. Then $D\uparrow x$ in $G$ if and only if $D\uparrow x$ in $E$.
\end{lemma}

Our first result deals with the order continuous extension of a positive compact (resp. weakly compact) operator defined on an order dense sublattice.

 \begin{theorem}\label{Theorem1}
Let $E$ and $F$ be Banach lattices, and let $G$ be an order dense sublattice of $E$. If $T: G\rightarrow F$ is  an order continuous positive compact (resp. weakly compact) operator, then $T$ extends uniquely to an order continuous positive compact (resp. weakly compact) operator $\widetilde{T}:E\rightarrow F$ with $\|\widetilde{T}\|=\|T\|$.

 Moreover, if, in addition, $T$ is a Riesz homomorphism, so is $\widetilde{T}$.
\end{theorem}
\begin{proof}
 We first assume that $T$ is an order continuous positive compact operator. For each $x\in E^{+}$, we claim that $\sup\{T(y):0\leq y\leq x, \,y\in G\}$ exists in $F$. Indeed, since $G$ is a sublattice and $T$ is  positive, $\{T(y):0\leq y\leq x, \,\,y\in G\}$ is an upwards directed  set, in other words, an increasing net, which is contained in a relatively compact subset of $F$. Therefore, $\{T(y):0\leq y\leq x, \,\,y\in G\}$ converges in norm to some element $u\in F^{+}$. It follows from \cite[Theorem 15.7]{Za} that $\{T(y):0\leq y\leq x, \,y\in G\}\uparrow u$ in $F$. For each $x\in E^{+}$, we define $\widetilde{T}(x)$ by $$\widetilde{T}(x)=\sup\{T(y):0\leq y\leq x, \,y\in G\}.$$It is easy to see that $\widetilde{T}(x)=T(x)$ for each $x\in G^{+}$.

 \par Now we claim that $T(x_{\alpha})\uparrow \widetilde{T}(x)$ holds in $F$ for every net $\{x_{\alpha}\}\subseteq G$ satisfying $0\leq x_{\alpha}\uparrow x$ in $E$. Evidently, $\widetilde{T}(x)$ is an upper bound of the net $\{T(x_{\alpha})\}$. If $v\in F$ is another upper bound of the net, we have to verify that $\widetilde{T}(x)\leq v$. To this end, let $y\in G^{+}$ be an arbitrary element such that $0\leq y\leq x$. Then we have $0\leq x_{\alpha}\wedge y\uparrow y$ in $E$. Since $G$ is order dense in $E$, it follows from Lemma \ref{Lemma1} that $0\leq x_{\alpha}\wedge y\uparrow y$ in $G$. Therefore, $T(x_{\alpha}\wedge y)\uparrow T(y)$ in $F$ since $T: G\rightarrow F$ is order continuous. Then, from $T(x_{\alpha}\wedge y)\leq T(x_{\alpha})\leq v$ we can easily see that $T(y)=\sup T(x_{\alpha}\wedge y)\leq v$, and hence $\widetilde{T}(x)=\sup\{T(y):0\leq y\leq x, \,y\in G\} \leq v$. Therefore, $T(x_{\alpha})\uparrow \widetilde{T}(x)$, as desired.
 \par Next, let $x,\,y\in E^{+}$. By the  order denseness of $G$ in $E$ there exist two nets $\{x_{\alpha}\}$ and $\{y_{\beta}\}$ in $G^{+}$ such that $0\leq x_{\alpha}\uparrow x$ and $0\leq y_{\beta}\uparrow y$ \,in $E$. Clearly,  $0\leq x_{\alpha}+y_{\beta}\uparrow x+y$ in $E$. Therefore, we have $ T(x_{\alpha})+T(y_{\beta})=T( x_{\alpha}+y_{\beta})\uparrow \widetilde{T}(x+y).$  From $T( x_{\alpha})\uparrow \widetilde{T}x$ and $T y_{\beta}\uparrow \widetilde{T}y$ it follows easily that $\widetilde{T}(x+y)=\widetilde{T}(x)+\widetilde{T}(y)$ for all $x,\,y\in E^{+}$. This implies that $\widetilde{T}$ is additive on $E^{+}$. Thus, $\widetilde{T}$ extends to a positive operator on the entire space $E$, denoted again by $\widetilde{T}$ (cf. \cite[Lemma 20.1]{Za}). To show that the  operator $\widetilde{T}:E\rightarrow F$ defined above is a compact operator, let $B_{+}=\{x\in E^{+}:\,\|x\|\leq 1 \}$ and $U_{+}=\{x\in G^{+}:\,\|x\|\leq 1 \}$. It should be noted  that for every $x\in B_{+}$ we have $$\widetilde{T}(x)\in \overline{\{T(y):0\leq y\leq x,\, y\in G\}}\subseteq \overline{T(U_{+})},$$ that is, $\widetilde{T}(B_{+})$ is contained in a compact subset $\overline{T(U_{+})}$ since $T$ is a positive compact operator. This implies that $\widetilde{T}$ is a compact operator with $\|\widetilde{T}\|=\|T\|$.
\par Finally, we have to show that the thus defined $\widetilde{T}:E\rightarrow F$ is order continuous on $E$. Let $0\leq x_{\alpha}\uparrow x$ in $E$ and let $D=\{y\in G^{+}: 0\leq y\leq x_{\alpha}\, \textmd{for some}\,\, \alpha\}$. Clearly, $D$ is upwards directed. In view of the denseness of $G$ in $E$, it follows that $D\uparrow x$ in $E$. Thus, $T(D)\uparrow \widetilde{T}(x)$ in $F$. On the other hand, since $\widetilde{T}$ is positive, we see that for any $y\in D$ there exists some index $\alpha$ with $y\leq x_{\alpha}$ such that
$$T(y)=\widetilde{T}(y)\leq \widetilde{T}(x_{\alpha})\leq \widetilde{T}(x).$$It follows  that $\widetilde{T}(x_{\alpha})\uparrow\widetilde{T}(x)$, that is, $\widetilde{T}$
is an order continuous positive compact extension of $T$ to $E$. It is easy to see that $T$ has a uniqie order continuous extension to $E$ since $G$ is order dense in $E$.
\par Now we assume that $T: G\rightarrow F$ is  an order continuous compact Riesz homomorphism. Then, by the above discussion, $T$ has a uniqie order continuous positive compact extension, $\widetilde{T}$, defined via the formula   $$\widetilde{T}(x)=\sup\{T(y):0\leq y\leq x, \,y\in G\}$$for each $x\in E^{+}$. To see that $\widetilde{T}$ is a Riesz homomorphism, let $x,\,y\in E$ with $x\wedge y=0$. By the order denseness of $G$ in $E$, pick two nets $\{x_{\alpha}\}$ and $\{y_{\beta}\}$ in $G^{+}$ such that $0\leq x_{\alpha}\uparrow x$ and $0\leq y_{\beta}\uparrow y$ in $E$. Since $T(x_{\alpha})\wedge T(y_{\beta})= T(x_{\alpha}\wedge y_{\beta})=0$, from
 $$T(x_{\alpha})\wedge T(y_{\beta})\uparrow\widetilde{T}(x)\wedge\widetilde{T}(y)$$we can easily see that $\widetilde{T}(x)\wedge\widetilde{T}(y)=0$, and hence $\widetilde{T}$ is also a Riesz homomorphism.

\par We note that the above proof is also valid for order continuous  weakly compact operators. In this case $\overline{T(U_{+})}$ is weakly compact and the upwards directed set $\{T(y):0\leq y\leq x, \,\,y\in G\}$ is contained in a weakly compact subset. It followed that  $\{T(y):0\leq y\leq x, \,\,y\in G\}$ weakly conveges, and hence it converges in norm (see, e.g.,  \cite[p. 89]{HH}).
\end{proof}
It is well known that every norm bounded positive linear functional on a sublattice of a Banach lattice has a norm preserving positive extension (cf. \cite[Theorem 39.2]{Za}). As an immediate consequence of the preceding theorem we have the following extension result for order continuous functionals.

\begin{corollary}\label{corollary1}
Suppose that $G$ is an order dense sublattice of a Banach lattice $E$. Then every norm bounded positive order continuous linear functional $\psi$ on $G$ can be uniquely extended to an order continuous positive linear functional $\varphi$ on $E$ such that $\|\varphi\|=\|\psi\|$.

\end{corollary}
\begin{remark}\label{Remark1}
 In \cite{SZ} it was shown that every norm bounded  order continuous positive linear functional defined on an ideal $I$ of a Banach lattice $E$ has a norm preserving order continuous positive extension to all of $E$. Note that an order dense sublattice of a Banach lattice need not be an ideal, and vice versa.
\end{remark}
Let $E_{n}^{\ast}$ be the Banach lattice of all norm bounded order continuous linear functionals on a normed Riesz space $E$. Recall that if $E$ has order continuous norm, then every norm bound linear functional on $E$ is order continuous. The following result, parallel to Corollary 1.3 of \cite{SZ}, should be immediate from Corollary \ref{corollary1}.
\begin{corollary}\label{Corollary2}
Let $E$ be a Banach lattice. if $E$ has a nonzero order dense sublattice with order continuous norm, then $E_{n}^{\ast}$ is nontrivial.
\end{corollary}

\par Recall that an order bounded operator $T:E\rightarrow E$ on an Archimedean Riesz space $E$ is said to be an orthomorphism whenever $x\bot y$ implies $Tx\bot y$ in $E$. It should be noted that every orthomorphism on an Archimedean Riesz space is order continuous (cf. \cite[Theorem 8.10]{2}). A well known result on the extension of orthomorphisms is that every orthomorphism on an Archimedean Riesz space $E$ always extends to an orthomorphism on $E^{\delta}$, the Dedekind completion of $E$ . By  Theorem \ref{Theorem1} we have the following  extension property of orthomorphisms.
\begin{corollary}\label{Corollary3}
Let $E$ be a Banach lattice and let $G$ be an order dense sublattice of $E$. Let $T:G\rightarrow G$ be a positive orthomorphism such that $T(U_{+})$ is norm totally bounded (resp. relatively weakly compact) subset of $E$, where $U_{+}=\{x\in G^{+}:\,\|x\|\leq 1 \}$. Then $T$ extends uniquely to a compact (resp. weakly compact) positive orthomorphism on $E$.
\end{corollary}
\begin{proof}
Clearly, if $T$ is regarded as an operator from $G$ into the Banach lattice $E$,  then $T:G\rightarrow E$ is a compact (resp. weakly compact) positive operator. We claim that $T:G\rightarrow E$ is oder continuous. Indeed, if $0\leq x_{\alpha}\uparrow x$ in $G$, then $0\leq Tx_{\alpha}\uparrow Tx$ in $G$ since $T$ is a positive orthomorphism on $G$. Thus, it follows from Lemma \ref{Lemma1} that $0\leq Tx_{\alpha}\uparrow Tx$ in $E$. This implies that $T:G\rightarrow E$ is an order continuous compact (resp. weakly compact) positive operator. Then, by Theorem \ref{Theorem1}, we know that the formula $$\widetilde{T}(x)=\sup\{T(y):0\leq y\leq x, \,y\in G\},\,\,\, x\in E^{+}$$defines a unique norm preserving order continuous positive extension of $T$ to $E$, which is also compact (resp. weakly compact). To complete the proof we have to show that $\widetilde{T}:E\rightarrow E$ is an orthomorphism. To this end, let $x,\,y\in E$ with $x\wedge y=0$. Since $G$ is order dense in $E$, we can pick two nets $\{x_{\alpha}\}$ and $\{y_{\beta}\}$ in $G^{+}$ such that $0\leq x_{\alpha}\uparrow x$ and $0\leq y_{\beta}\uparrow y$ in $E$. Since $T:G\rightarrow G$ is a positive orthomorphism we have that $T(x_{\alpha})\wedge y_{\beta}=0$ in $G$. From
 $$T(x_{\alpha})\wedge y_{\beta}\uparrow\widetilde{T}(x)\wedge y$$we can easily see that $\widetilde{T}(x)\wedge y=0$. This implies that $\widetilde{T}$ is likewise a positive orthomorphism on $E$.
\end{proof}
An ideal of a Banach lattice need not be order dense. The following result deals with  the extension of order continuous positive compact (resp. weakly compact) operators defined on an ideal of a Banach lattice. It generalizes the result of  \cite{SZ}, mentioned above in Remark \ref{Remark1}.

\begin{theorem}\label{Theorem2}
Suppose that $E,\,\,F$  are Banach lattices and $I$ is an ideal of $E$. Then any order continuous positive compact (\,resp. weakly compact\,) operator $T:I\rightarrow F$ has a norm preserving positive compact  (\,resp. weakly compact\,) extension to $E$, which is likewise order continuous on $E$.
\end{theorem}
\begin{proof}
From \cite[Theorem 1.2]{Z} we know that $T:I\rightarrow F$ has a norm preserving positive compact (resp. weakly compact) extension to all of $E$ satisfying $$\widetilde{T}(x)=\sup\{T(y):0\leq y\leq x, \,y\in I\},\,\,\, x\in E^{+}.$$To complete the proof  we have to establish the order continuity of  $\widetilde{T}$. Let $0\leq x_{\alpha}\uparrow x$ in $E$. Since $\widetilde{T}$ is positive, we have  $\widetilde{T}(x_{\alpha})\leq \widetilde{T}(x)$ for each $\alpha$. Suppose that $v\in F$ is an arbitrary upper bound of $\{\widetilde{T}(x_{\alpha})\}$. We only have to show that $\widetilde{T}(x)\leq v$. To this end, let $y\in I$ satisfy $0\leq y\leq x$. Clearly, $0\leq x_{\alpha}\wedge y\uparrow y$ in $E$. Since $I$ is an ideal of $E$, we have $x_{\alpha}\wedge y\in I$. From Lemma \ref{Lemma1} it easily follows that $0\leq x_{\alpha}\wedge y\uparrow y$ in $I$. By the order continuity of $T$ on $I$ we can see that $0\leq T(x_{\alpha}\wedge y)\uparrow T(y)$ holds in $F$. From $$T(x_{\alpha}\wedge y)=\widetilde{T}(x_{\alpha}\wedge y)  \leq \widetilde{T}(x_{\alpha})\leq v$$it follows that $T(y)\leq v$. This implies that $v$ is an upper bound of the set $\{T(y):0\leq y\leq x, \,y\in I\}$, hence, $\widetilde{T}(x)\leq v$, as desired. Now we can easily conclude that $\widetilde{T}(x)$ is the supremum  of $\{\widetilde{T}(x_{\alpha})\}$, that is , $\widetilde{T}(x_{\alpha})\uparrow\widetilde{T}(x)$ .
\end{proof}
For a Banach lattice $E$, we let $$E^{\,a}:=\{x\in E: \textmd{every monotone sequence in}\,\, [\,0,\,\, |\,x\,|\,]\,\, \textmd{is convergent}\}.$$Recall that $E^{\,a}$ is the maximal ideal in $E$ on which the induced norm is order continuous (cf. \cite[Proposition 2.4.10\,]{M}). As a consequence of the preceding theorem we have the following result.
\begin{corollary}\label{Corollary4}
 Let $T:E\rightarrow F$ be a positive compact (resp. weakly compact) operator between two Banach lattices. Then there exists an order continuous positive compact (resp. weakly compact) operator $\widetilde{T}:E\rightarrow F$ such that $0\leq\widetilde{T}\leq T$ and $\widetilde{T}=T$ on $E^{\,a}$.
\end{corollary}
\begin{proof}
We denote the restriction of $T$ to $E^{\,a}$ as $T|_{E^{\,a}}$. Clearly, $T|_{E^{\,a}}:E^{\,a}\rightarrow F$ is a positive compact (resp. weakly compact) operator. Since $E^{\,a}$ is the maximal ideal with the induced norm order continuous, $T|_{E^{a}}:E^{\,a}\rightarrow F$ is order continuous. Now, from Theorem \ref{Theorem2} we know that $T|_{E^{a}}:E^{\,a}\rightarrow F$ has an order continuous positive compact (resp. weakly compact) extension to $E$, say, $\widetilde{T}$, defined via the formula
$$\widetilde{T}(x)=\sup\{T|_{E^{\,a}}(y):0\leq y\leq x, \,y\in E^{\,a}\},\,\,\, x\in E^{+}.$$It remains  to show that $\widetilde{T}\leq T$. To this end, let $x\in E^{+}$. Then for each $y\in E^{\,a}$ with $0\leq y\leq x$, from
$$T|_{E^{\,a}}(y)=T(y)\leq T(x)  $$we can easily see that $\widetilde{T}(x)\leq T(x)$. This implies that $0\leq\widetilde{T}\leq T$.
\end{proof}

\end{document}